\documentclass[11pt,english]{amsart}
\usepackage{amssymb,latexsym}
\usepackage{amsfonts, amsmath, amsthm, amssymb, hyperref, fancybox, graphicx, color, times, babel}

\newtheorem {thm}{Theorem}
\newtheorem* {thm*}{Theorem}

\newtheorem {cor}[thm]{Corollary}

\newtheorem* {cor*}{Corollary}
\newtheorem {lem}[thm]{Lemma}

\newtheorem {prop}[thm]{Proposition}
\newtheorem {defi}[thm]{Definition}
\newtheorem {rem}[thm]{Remark}

\theoremstyle{definition}
\newtheorem {exa}[thm]{Example}
\newtheorem* {conj*}{Conjecture}

\newtheorem* {quest*}{Question}

\DeclareMathOperator{\Gal}{Gal}

\DeclareMathOperator{\tors}{tors}
\DeclareMathOperator{\dens}{dens}

\newcommand{\Q}{\mathbb{Q}}

\newcommand{\B}{B}

\newcommand{\p}{\mathfrak{p}}

\newcommand{\OO}{\mathcal{O}}
\newcommand{\Ktors}{K^\times_{\text{tors}}}

\newcommand{\D}{D_{K,\ell}}

\begin{document}

\title[Reductions of subgroups of the multiplicative group]{Reductions of subgroups of the multiplicative group}
\author{Christophe~Debry \and Antonella~Perucca}
\address[]{Mathematics Department, KU Leuven, Celestijnenlaan 200B, 3001 Leuven, Belgium
\emph{and} Korteweg--de Vries Institute for Mathematics, Universiteit van Amsterdam.}
\email[]{christophe.debry@wis.kuleuven.be}
\address[]{Fakult\"at Mathematik, Universit\"at Regensburg, 93040 Regensburg, Germany}
\email[]{antonella.perucca@mathematik.uni-regensburg.de}
\thanks{}
\maketitle

\begin{abstract}
Let $K$ be a number field, and let $G\subset K^\times$ be a finitely generated subgroup. Fix some prime number $\ell$, and consider the set of primes $\mathfrak p$ of $K$ satisfying the following property:  the reduction of $G$ modulo $\p$ is well-defined and has size coprime to $\ell$. We give a closed--form expression for the density of this set.
\end{abstract}

\section{Introduction}

\subsection{Motivation and Aim}
Consider the number field $K=\mathbb{Q}(\zeta_8)$ and let $G$ be one of the following subgroups of $K^\times$: $$\langle 12,18\rangle\qquad \langle 12\zeta_8,18\rangle \qquad \langle 12,18\zeta_8\rangle\,.$$
For the primes $\p$ of $K$ which do not lie over $2$ or $3$, we consider the reduction of $G$ modulo $\p$. What is the density of the set of primes $\p$ for which the size of $(G \bmod \p)$ is odd? It turns out that this density is $\frac{1}{56}$ for the first two groups and $\frac{1}{448}$ for the third.

The object of this paper is the computation of these densities and understanding why we get these very numbers.
For every number field $K$, for every finitely generated subgroup $G\subset K^\times$ and for every prime number $\ell$ we study the set of primes $\p$ of $K$ satisfying the following property: the reduction of $G$ modulo $\p$ is well-defined and has size coprime to $\ell$. This set turns out to have a natural density, for which we were able to write down a closed--form expression.\\
The problem that we solve in this paper has been opened by the works of Hasse \cite{Hasse1, Hasse2} in the 1960's. For an exhaustive survey about related questions we refer the reader to~\cite{MoreeArtin} by Moree. Recently the problem was solved for rank $1$ by the second--named author~\cite{PeruccaKummer}.
The generalization of this problem for algebraic groups has been studied by Jones and Rouse in~\cite{Jones_Rouse} for rank $1$ (in the generic case, i.e. assuming that the degrees of the torsion fields and the Kummer extensions appearing in the formulas are maximal). A complete solution to the problem for elliptic curves seems for the moment out of reach, however this provides a direction for future research.

\subsection{Illustration of the results}

Let $K$ be a number field. Let $G$ be a finitely generated and torsion--free subgroup of $K^\times$ of positive rank. We always tacitly exclude the finitely many primes $\p$ of $K$ for which the reduction of $G$ modulo $\p$ is not well-defined or contains $(0 \bmod \p)$.  If $\ell$ is a prime number, we want to compute the density 
$$\D(G):=\dens\{ \p: \ell\nmid \sharp(G \bmod \p) \}\,.$$
Let $k^\times_\p$ denote the multiplicative group of the residue field at $\p$. The size of $k^\times_\p$ is prime to $\ell$ if $\p$ does not split completely in $K_{\ell}:=K(\zeta_\ell)$. By the Chebotarev Density Theorem we then have
$$\D(G)\geqslant 1-\frac{1}{[K_\ell:K]}\,.$$
Suppose that $\ell$ is odd or that $\zeta_4\in K$, which ensures that the cyclotomic extensions $K_{\ell^n}:=K(\zeta_{\ell^n})$ are cyclic for every $n\geqslant 0$. Moreover, suppose that the elements of $G$ are $\ell$--th powers in $K$ only if they are already $\ell$-th powers in $G$. In this generic case, what matters about $G$ is only its rank: 

\begin{thm}\label{theorem1} 
Let $K$ be a number field and let $G$ be a torsion--free subgroup of $K^\times$ of finite rank $r>0$. Let $\ell$ be a prime number and assume that $\ell$ is odd or $\zeta_4\in K$.
Let $t\geqslant 1$ be the largest integer such that $K_\ell = K_{\ell^t}$. If the condition $G\cap (K^\times)^\ell = G^\ell$ holds then we have 
$$\D(G) =  1-\frac{1}{[K_\ell:K]}+\frac{\ell-1}{[K_\ell:K]\cdot(\ell^{r+1}-1)\cdot \ell^{r(t-1)}}\,.$$
\end{thm}

More generally we prove the following:

\begin{thm}\label{theorem2}
Let $K$ be a number field and let $G$ be a torsion--free subgroup of $K^\times$ of finite rank $r>0$. Let $\ell$ be a prime number and assume that $\ell$ is odd or $\zeta_4\in K$. 
Let $t\geqslant 1$ be the largest integer such that $K_\ell = K_{\ell^t}$.

\begin{itemize}
	\item[(a)] There is some smallest integer $\tau$ satisfying $\tau\geq 1$ and $\zeta_{\ell^{\tau+1}}\not\in K_{\ell^{\tau}}\left(\sqrt[\ell^{\tau}]{G}\right)$. 
	
	\item[(b)] There exist unique integers $d_1\leqslant\cdots\leqslant d_r$ such that for every sufficiently large integer $n$ we have $$\frac{G}{G\cap (K^\times)^{\ell^n}} \cong \bigoplus_{i=1}^r  \frac{\mathbb{Z}}{\ell^{n-d_i}\mathbb{Z}}\,.$$ 
	\item[(c)] Write $\tau_i = \max(\tau,d_i)$ for all $i=1,\ldots,r$. Then $\D(G)$ equals $$1 - \frac{1}{[K_\ell:K]} + \frac{\ell^{t-1}(\ell-1)}{[K_\ell:K]}\left(\frac{\ell^{-\tau}}{1-\ell^{-1}}-\sum_{i=1}^r \ell^{d_1+\cdots+d_i-i\tau_i}\left(\frac{\ell^{-d_i}}{1-\ell^{-i}}-\frac{\ell^{-\tau_i}}{1-\ell^{-i-1}}\right)\right)\,.$$
\end{itemize}
\end{thm}

In particular $\D(G)$ is a rational number that may be explicitly computed, see Section~\ref{section-computations}. As described in Section~\ref{heuristics}, the formula for the density may be exactly recovered by heuristics, if we suppose that the $\ell$--part of the reductions of points behave ``randomly''.
 
In the above Theorem we assume that $G$ is torsion--free, but we can also handle the case in which $G$ has torsion, because roots of unity of order prime to $\ell$ do not affect the density while if $G$ contains a root of unity of order $\ell$ then $\D(G)=0$ holds. Concerning the possibilities for $\ell$ and $K$, there is only one case left, namely $\ell = 2$ and $\zeta_4\notin K$. We deal with this remaining case as follows:
	
\begin{thm}\label{theorem3}
Let $K$ be a number field and let $G$ be a torsion--free subgroup of $K^\times$ of finite rank $r>0$. 
If $\zeta_4\notin K$, we have 
$$D_{K,2}(G)=\frac{1}{2} D_{K_4,2}(G)+  \frac{c }{2} \cdot [K_4(\sqrt{G}):K_4]^{-1}$$
where $c=0$ if $G$ contains minus a square in $K^\times$ and $c=1$ otherwise.
\end{thm}

A result of independent interest is a general formula for the degree of the Kummer extension $K_{\ell^m}(\sqrt[\ell^n]{G})/K_{\ell^m}$ for $m\geq n$, in terms of finitely many integer parameters that express the $\ell$--divisibility of $G$ in $K$, see Section~\ref{section-degree}. 

\subsection{Overview of the paper} 

The paper is self-contained because it relies only on classical results about number fields.
Section~\ref{section-preliminaries} contains   preliminaries.
In Section~\ref{section-existence} we show that the natural density $\D(G)$ exists and that it may be expressed as an infinite sum:
$$\D(G)=\sum_{n\geqslant 0}\;\big(\; [K_{\ell^n}(\sqrt[\ell^{n}]{G}):K]^{-1}-[K_{\ell^{n+1}}(\sqrt[\ell^{n}]{G}):K]^{-1}\; \big)\,.$$
In Section~\ref{section-parameters} we determine a parametric formula for the size of the quotient of $G$ by the elements that are $\ell^n$--th powers in $K^\times$. In Section~\ref{section-degree} we may then express the extension degrees that appear in the above summation formula for the density. In Section~\ref{section-proof} we obtain a closed--form expression for $\D(G)$.
Section~\ref{specialheuristics} contains heuristics and special cases. Finally in Section~\ref{section-computations} we discuss the computability of the density.
Theorems~\ref{theorem2}~and~\ref{theorem3} are  proven in Section~\ref{section-proof} while Theorem~\ref{theorem1} is proven in Section~\ref{specialcases}.

\subsection{Acknowledgements} 
The first author is supported by a PhD fellowship of the Research Foundation -- Flanders (FWO). The second author is a lecturer at the University of Regensburg.

\subsection{Notation} 

In the whole paper (and without further mention), $K$ is a number field and $\ell$ is a fixed prime number. We use the following standard notation: $K^\times$ is the multiplicative group of $K$, and $\Ktors$ is the torsion subgroup of $K^\times$, consisting of roots of unity. For any nonzero prime ideal $\p$ of the ring of integers $\mathcal{O}_K$ of $K$, the residue field $\mathcal{O}_K/\p$ will be denoted by $k_\p$.
For $n\geqslant 0$, we write $\mu_{\ell^n}$ for the group of $\ell^n$--th roots of unity in a fixed algebraic closure $\bar{K}$ of $K$, and we call $\mu_{\ell^\infty}$ the union of all $\mu_{\ell^n}$.  We write $K_{\ell^n}$ for the cyclotomic extension $K(\mu_{\ell^n})$ of $K$, and we call $K_{\ell^\infty}$ the union of all $K_{\ell^n}$. If $G$ is a  subgroup of $K^\times$  and if $m\geqslant n$ we write $K_{\ell^m}(\sqrt[\ell^n]{G})$ for the smallest extension of  $K_{\ell^m}$ containing some $\ell^n$--th root of every element of $G$ (equivalently, all $\ell^n$--th roots). Moreover, we denote by $G^{\ell^n}$ the subgroup of $K^\times$ consisting of the $\ell^n$--th powers of elements of $G$. As customary, we write $\langle a_1,\ldots,a_n \rangle$ for the group generated by the elements $a_1,\ldots,a_n$. The $\ell$--adic valuation on $\mathbb{Z}$ will be denoted by $v_\ell$, and the size of a set by $\sharp$.

\section{Divisibility and independence properties}\label{section-preliminaries} 

\subsection{Divisibility properties in $K^\times$} 

An element of $K^\times$ is said to be $\ell$--divisible if it has some $\ell$--th root in $K^\times$. We consider this divisibility modulo roots of unity:

\begin{defi}[Strongly $\ell$--indivisible]
We call $a\in K^\times$ \emph{strongly $\ell$--indivisible} if there is no root of unity $\zeta\in K\cap \mu_{\ell^\infty}$ such that $a\zeta \in (K^\times)^{\ell}$.
\end{defi}

The ambient field is important in this definition but we do not specify it as soon as it is clear from the context.
We could drop the assumption that the order of $\zeta$ is a power of $\ell$ because roots of unity in $K$ of order prime to $\ell$ belong to  $(K^\times)^{\ell}$.

\begin{lem}\label{rem-indivisible}
If $a\in K^\times$ is strongly $\ell$--indivisible and $m>n\geqslant 0$ are integers, then $a^{\ell^{n}}\not\in (K^\times)^{\ell^m}$.
\end{lem}
\begin{proof}
If we have $a^{\ell^{n}} = b^{\ell^m}$ for some $b\in K^\times$ then $\zeta=a^{-1}\cdot b^{\ell^{m-n}}$ is a root of unity in $K$ of order dividing $\ell^n$ and satisfying $a\zeta\in (K^\times)^{\ell}$, contradicting that $a$ is strongly $\ell$--indivisible.
\end{proof}
 
\begin{lem}\label{convenience} If $a\in K^\times$ is not a root of unity then $a=A^{\ell^d}\zeta$ for some nonnegative integer $d$, for some $A\in K^\times$ that is strongly $\ell$--indivisible and for some $\zeta\in K\cap \mu_{\ell^\infty}$. The integer $d$ is uniquely determined by the triple $(K,\ell,a)$.\end{lem}
\begin{proof}
This is a special case of~\cite[lem.~8]{PeruccaKummer}, which we include for the convenience of the reader. For any $x\in K^\times$ which is not a root of unity, there exists some maximal $d_x$ such that $x\in (K^\times)^{\ell^d}$. Indeed, if $x\in \mathcal{O}_K^\times$ this follows from the fact that $\mathcal{O}_K^\times$ is finitely generated, and otherwise this follows from the unique ideal factorisation of the principal fractional ideal $x\mathcal{O}_K$. Since $K\cap \mu_{\ell^\infty}$ is finite, we can define $d = \max\{d_{a\zeta}\mid \zeta\in K\cap \mu_{\ell^\infty}\}$. Taking $\zeta$ such that $d = d_{a\zeta}$ yields an element $A\in K^\times$ such that $a = \zeta^{-1}A^{\ell^d}$. By maximality of $d$, the element $A$ is strongly $\ell$--indivisible. The uniqueness of $d$ follows from Lemma~\ref{rem-indivisible}.
\end{proof}

\subsection{Kummer theory}\label{sec:kummercyclo}

We first recall some basic facts about Kummer theory, standard references being for example \cite{Birch, Lang}. Let $n$ be a nonnegative integer and suppose that $K$ contains the $\ell^n$--th roots of unity. The extension $K(\sqrt[{\ell^n}]{G})/K$ is Galois and the map $$\Gal(K(\sqrt[{\ell^n}]{G})/K)\times G\cdot (K^\times)^{{\ell^n}}\to \mu_{\ell^n}:(\sigma,a)\mapsto \sigma(\sqrt[{\ell^n}]{a})/\sqrt[{\ell^n}]{a}$$ where $\sqrt[{\ell^n}]{a}$ is any ${\ell^n}$--th root of $a$, is well--defined, bilinear, injective on the left and has kernel $(K^\times)^{\ell^n}$ on the right. We hence have the following group isomorphisms (of which the middle one is non--canonical): $$\Gal(K(\sqrt[{\ell^n}]{G})/K) \cong \text{Hom}_{\mathbb{Z}}\left(\frac{G\cdot (K^\times)^{{\ell^n}}}{(K^\times)^{\ell^n}},\mu_{\ell^n}\right)  \cong \frac{G\cdot (K^\times)^{{\ell^n}}}{(K^\times)^{\ell^n}}\cong \frac{G}{G\cap (K^\times)^{\ell^n}}\,.$$ In particular, the degree $[K(\sqrt[{\ell^n}]{G}):K]$ is equal to the index of $G\cap (K^\times)^{\ell^n}$ in $G$ and hence is a power of $\ell$. Another consequence is the following: if for some $a$ and $b$ in $K^\times$ we have $K(\sqrt[{\ell^n}]{a})\subseteq K(\sqrt[{\ell^n}]{b})$ then $a = x^{\ell^n}b^t$ for some $x\in K^\times$ and for some integer $t$.

We also recall some basic cyclotomic theory. Let $t\geqslant 1$ be the largest integer for which $K_\ell = K_{\ell^t}$ and suppose that $t\geqslant 2$ if $\ell = 2$. Then any cyclotomic extension of $K$ is cyclic and for any $m\geqslant t$ the group $\Gal(K_{\ell^m}/K_\ell)\cong \{x+\ell^m\mathbb{Z}\mid x\equiv 1\mod{\ell^t\mathbb{Z}}\}$ is cyclic of order $\ell^{m-t}$. In particular, $[K_{\ell^m}:K] = \ell^{m-t}[K_\ell:K]$ and $K_{\ell^{m+1}}\neq K_{\ell^m}$, and hence $v_\ell(\sharp(K_{\ell^m}^\times)_{\tors}) = m$. We also infer that any subextension of $K_{\ell^m}\supset K_\ell$ is cyclotomic.

\begin{prop}\label{prop:stay-indiv}
Let $a\in K^\times$ be strongly $\ell$--indivisible.
\begin{itemize}
	\item[(a)] (Schinzel) If $\zeta_\ell\notin K$ then the extension $K_\ell(\sqrt[\ell]{a})/K$ is not abelian.
	\item[(b)] Suppose that $\ell$ is odd or that $\zeta_4\in K$. For any nonnegative integer $m$, the element $a$ is also strongly $\ell$--indivisible in $K_{\ell^m}$.
\end{itemize}
\end{prop}
\begin{proof}[Proof of (a):] This is a special case of Schinzel's Theorem about abelian radical extensions, see  \cite{Lenstra, Schinzel}, that we include for the convenience of the reader. Let $\mathcal G=\Gal(K_\ell(\sqrt[\ell]{a})/K)$, and suppose that $\mathcal G$ is abelian. Since $\zeta_\ell\notin K$, there is some automorphism $\sigma\in \mathcal G$ which is not the identity on $K_{\ell}$. We then have $\sigma(\zeta_{\ell})=\zeta_{\ell}^c$ for some integer $c$ such that $1-c$ is prime to $\ell$. Fix some $\alpha\in K_\ell(\sqrt[\ell]{a})$ satisfying $\alpha^\ell = a$. \newline Let $\tau$ be an element of $\mathcal G$. Since $\tau(\alpha)/\alpha$ is an $\ell$--th root of unity, it is raised to the $c$--th power by $\sigma$ so we have $\sigma(\tau(\alpha))\alpha^c = \sigma(\alpha)\tau(\alpha)^c$. Since $\sigma\tau = \tau\sigma$, we deduce $\tau(\sigma(\alpha)\alpha^{-c}) = \sigma(\alpha)\alpha^{-c}$. This is true for every $\tau$ so $\sigma(\alpha)\alpha^{-c}$ is in $K$ hence its $\ell$--th power $a^{1-c}$ is an $\ell$--th power in $K^\times$, contradicting that $a$ is strongly $\ell$--indivisible.
 
\emph{Proof of (b):} Suppose that $a$ is not strongly $\ell$--indivisible in $K_{\ell^m}$. Then $K_{\ell^m}(\sqrt[\ell]{a})$ is a cyclotomic extension of $K$ and in particular it is abelian, so we deduce from (a) that $\zeta_\ell\in K$. The condition on $\ell$ ensures that all subextensions of $K_{\ell^m}(\sqrt[\ell]{a})/K$ are again cyclotomic. Since $1<[K(\sqrt[\ell]{a}):K]\leqslant \ell$ we must have $K(\sqrt[\ell]{a}) = K_{\ell^{t+1}}$, where $t\geq 1$ is the largest integer such that $K = K_{\ell^t}$. So $K(\sqrt[\ell]{a}) = K(\sqrt[\ell]{\zeta_{\ell^t}})$ and hence $a = x^{\ell}\zeta_{\ell^t}^n$ for some $x\in K^\times$ and for some integer $n$, which contradicts that $a$ is strongly $\ell$--indivisible in $K_{\ell^m}$. 
\qedhere
\end{proof}

The fourth--roots of a strongly $2$-indivisible element of $K^\times$ are not contained in $K_{2^\infty}$, but if $\zeta_4\notin K$ it can happen that the square--roots are. See~\cite[section 4.3]{PeruccaKummer} for a precise description of this phenomenon.

\subsection{Independence properties in $K^\times$} 
We generalize the notion of strongly $\ell$--indivisible to several points as follows:

\begin{defi}[Strongly $\ell$--independent] We say that  $a_1,\ldots,a_r\in K^\times$ are {strongly $\ell$--independent} if $a_1^{x_1}\cdots a_r^{x_r}$ is strongly $\ell$--indivisible whenever $x_1,\ldots,x_r$ are integers not all divisible by $\ell$. \end{defi}

Strongly $\ell$--independent elements are each strongly $\ell$--indivisible. By Proposition~\ref{prop:stay-indiv}, if $\ell$ is odd or $\zeta_4\in K$, then elements of $K$ that are strongly $\ell$--independent remain strongly $\ell$--independent in $K_{\ell^m}$ for every $m\geqslant 0$.

\begin{lem}\label{lem:power-indep}
Suppose that $\ell$ is odd or that $\zeta_4\in K$. Let $a_1,\ldots,a_r$ be strongly $\ell$--independent elements of $K^\times$ and let $x_1,\ldots,x_r$ and $n\geqslant 0$ be integers. If $a_1^{x_1}\cdots a_r^{x_r}\in (K^\times)^{\ell^n}$ then $x_1,\ldots,x_r$ are all divisible by $\ell^n$.
\end{lem}
\begin{proof}Let $a=a_1^{x_1}\cdots a_r^{x_r}\in (K^\times)^{\ell^n}$ and let $e = \min_i(v_\ell(x_i))$ (we suppose that not all $x_i$ are zero because the statement is trivial in that case). Then some integer $x_i\ell^{-e}$ is not divisible by $\ell$.  Since $a_1,\ldots,a_r$ are strongly $\ell$--independent, we deduce that  $a = b^{\ell^e}$ for some strongly $\ell$--indivisible element $b$. Lemma~\ref{rem-indivisible} now implies that $e\leqslant n$.
\end{proof}

\section{Parameters describing the $\ell$-divisibility of $G$} \label{section-parameters}

Let $G$ be a finitely generated subgroup of $K^\times$. For $n\geqslant 0$, we denote by $G_n := G\cap (K^\times)^{\ell^n}$ the subgroup of $G$ consisting of those elements that are $\ell^n$--th powers in $K^\times$. The aim of this section is to describe the group structure of the finite abelian $\ell$--group $G/G_n$.

\subsection{Choosing a good basis}

\begin{lem} \label{lem:fg} Let $\mathcal{G}$ be a finitely generated subgroup of $K^\times$. Then there exists a positive integer $m$ with the following property: if $n\in\mathbb{Z}_{\geqslant 0}$ and $x\in K^\times$ satisfy $x^{\ell^n}\in\mathcal{G}$, then $x^{\ell^m}\in\mathcal{G}$.
\end{lem}

\begin{proof}
Consider the subgroup $\mathcal{H}$ of $K^\times$ consisting of those $x$ satisfying $x^{\ell^n}\in \mathcal{G}$ for some $n\geq 0$.
We claim that $\mathcal{H}$ is finitely generated. Indeed, let $\{g_1,\ldots,g_r\}$ be a generating set for $\mathcal{G}$ and let $S$ be the set of prime ideals of $\mathcal{O}_K$ appearing in the ideal factorisations of the principal fractional ideals $g_1\mathcal{O}_K,\ldots,g_r\mathcal{O}_K$. For any $x\in \mathcal{H}$ the prime ideals appearing in the factorisation of $x\mathcal{O}_K$ lie in $S$, so we have a morphism $\mathcal{H}\to \mathbb{Z}^{\sharp S}$ whose kernel is $\mathcal{H}_0 := \mathcal{H}\cap \mathcal{O}_K^\times$. The groups $\mathcal{H}_0$ and $\mathcal{H}/\mathcal{H}_0 \cong \text{im}(\mathcal{H}\to \mathbb{Z}^{\sharp S})$ are subgroups of $\mathcal{O}_K^\times$ and $\mathbb{Z}^{\sharp S}$ respectively and hence are finitely generated.
Thus so is $\mathcal{H}$, say $\mathcal{H} = \langle h_1,\ldots,h_r\rangle$. 
Take a positive integer $m$ such that $h_1^{\ell^m},\ldots,h_r^{\ell^m}\in\mathcal{G}$. Then 
$x^{\ell^m}\in\mathcal{G}$ for all $x\in\mathcal{H}$, so $m$ is as desired.
\end{proof}

\begin{cor}\label{boundbasis}
Let $\mathcal{G}$ be a finitely generated and torsion--free subgroup of $K^\times$ and let $m$ be an integer as in Lemma \ref{lem:fg}. Then no element of a basis for $\mathcal{G}$ is an $\ell^{m+1}$--st power in $K^\times$.
\end{cor}

\begin{proof}
Let $\{b_1,\ldots,b_r\}$ be a basis for $\mathcal{G}$ and suppose that $b_1 = b^{\ell^{m+1}}$ for some $b\in K^\times$. By definition of $m$ we have $b^{\ell^m}\in \mathcal{G}$ so $b_1$ is an $\ell$-th power in $G$, and that is impossible.
\end{proof}

Recall that we fixed a finitely generated subgroup $G$ of $K^\times$, which we from now on assume to be torsion--free (and hence to be a free $\mathbb{Z}$--module). We write $z = v_\ell(\sharp\Ktors)$. To any $a\in K^\times$ which is not a root of unity we may associate by Lemma \ref{convenience} an integer $d(a)$ such that $a^{\ell^z} = A^{\ell^{z+d(a)}}$ for some strongly $\ell$--indivisible $A$. So to any basis $\mathcal{B}$ of $G$ one can associate the quantity $d(\mathcal{B}) = \sum_{b\in\mathcal{B}} d(b)$. 
Since $\{b^{\ell^z}\mid b\in\mathcal{B}\}$ is a basis for $G^{\ell^z}$ and each $b^{\ell^z}$ is of the form $B^{\ell^{z+d(b)}}$ for some $B\in K^\times$, Corollary \ref{boundbasis} implies that $z+d(b)<m+1$ for any $b\in \mathcal{B}$, where $m$ is  as in Lemma \ref{lem:fg} for $\mathcal{G} = G^{\ell^z}$. In particular, $d(\mathcal{B}) \leqslant r(m-z)$ for any basis $\mathcal{B}$ for $G$.

\begin{thm}\label{thm-def-parameters}
Let $G$ be a torsion--free subgroup of $K^\times$ of finite rank $r>0$. Then there are nonnegative integers $d_1,\ldots,d_r,h_1,\ldots,h_r$ such that 
$G = \langle \B_1^{\ell^{d_1}}\zeta_1,\ldots,\B_r^{\ell^{d_r}}\zeta_r \rangle$ for some strongly $\ell$--independent elements $\B_1,\ldots,\B_r$ of $K^\times$ and for some roots of unity $\zeta_i$ in $K$ of order $\ell^{h_i}$.
\end{thm}

We say that the $2r$--tuple $(d_1,\ldots,d_r,h_1,\ldots,h_r)$ of nonnegative integers \emph{expresses the $\ell$--divisibility of $G$ in $K$} if the integers $d_1,\ldots,h_r$ are as in the above theorem.

\begin{proof}
We can  choose a basis $\mathcal{B} = \{b_1,\ldots,b_r\}$ of $G$ maximizing the quantity $d(\mathcal{B})$. Take strongly $\ell$--indivisible $\B_i\in K^\times$ and nonnegative integers $d_i=d(b_i)$ such that $b_i^{\ell^z} = \B_i^{\ell^{z+d_i}}$. We claim that $\B_1,\ldots,B_r$ are strongly $\ell$--independent. Indeed, if they were not, we could take integers $x_1,\ldots,x_r$, not all divisible by $\ell$, such that the element $a = \B_1^{x_1}\cdots \B_r^{x_r}$  is not strongly $\ell$--indivisible. Let 
$d_j = \max\{d_i\mid x_i\not\in \ell\mathbb{Z}\}$ and assume without loss of generality that $x_j = 1$. The element $b_j' = \prod_i b_i^{x_i\ell^{d_j-d_i}}$ is equal to $a^{\ell^{d_j}}$ up to some root of unity, so the fact that $a$ is not strongly $\ell$--indivisible implies that $d(b_j') > d_j$. Also note that $b_j'/b_j$ is an element of the subgroup of $G$ generated by $\{b_i\mid i\neq j\}$, so $\mathcal{B}' = (\mathcal{B}\cup\{b_j'\})\setminus\{b_j\}$ is a basis for $G$ satisfying $d(\mathcal{B}') = d(\mathcal{B}) + d(b_j') - d(b_j) > d(\mathcal{B})$, contradicting our choice of $\mathcal{B}$. 
\end{proof}

\begin{exa}\label{single}
Let $a = A^{\ell^d}\zeta$ where $A\in K^\times$ is strongly $\ell$--indivisible, $d$ is a nonnegative integer and $\zeta\in K$ is a root of unity of order $\ell^h$. Then $(d,h)$ expresses the $\ell$--divisibility of $\langle a\rangle$ in $K$. If $\zeta$ has $\ell^d$-th roots in $K$, we may also write $a = B^{\ell^d}$ with $B\in K^\times$ strongly $\ell$--indivisible. Then $(d,0)$ also expresses  the $\ell$--divisibility of $\langle a\rangle$ in $K$. 
\end{exa}

\subsection{The group structure of $G/G_n$}

\begin{thm}\label{thm-GmodGn}
Let $G$ be a torsion--free subgroup of $K^\times$ of finite rank $r>0$ and denote $v_{\ell} (\sharp \Ktors)$ by $z$. We keep the notations of Theorem \ref{thm-def-parameters} and write $b_i = \B_i^{\ell^{d_i}}\zeta_i$.
Let $n$ be a nonnegative integer, and write $\delta_i = \max(n-d_i,0)$. Let $G_n=G\cap (K^\times)^{\ell^n}$ and let
 $H_n$ be the subgroup of $G/G_n$ generated by the classes of $b_1^{\ell^{\delta_1}},\ldots,b_r^{\delta_r}$. Then $H_n$ is a finite cyclic $\ell$--group and we have 
$$v_{\ell} (\sharp H_n) =  \max(h_1-\delta_1,\ldots,h_r-\delta_r,z-n,0)+\min(n-z,0)\,.$$
Moreover, we have $$\frac{G/G_n}{H_n}\cong\bigoplus_{i=1}^r \frac{\mathbb{Z}}{\ell^{\delta_i}\mathbb{Z}}.$$
\end{thm}
\begin{proof}
To study the structure of $H_n$ we identify it with the subgroup of $G\!\cdot\!(K^\times)^{\ell^n}/(K^\times)^{\ell^n}$ which is generated by the classes of the $b_i^{\ell^{\delta_i}}$. 
Since $\delta_i\geqslant n-d_i$, we have $b_i^{\ell^{\delta_i}}\equiv \zeta_i^{\ell^{\delta_i}}$ modulo $(K^\times)^{\ell^n}$. Each $\zeta_i^{\ell^{\delta_i}}$ is a power of $\zeta :=\zeta_j^{\ell^{\delta_j}}$, where $j$ is an index for which $h_j-\delta_j$ is maximal.
Then $H_n$ is generated by the class of $\zeta$ so in particular it is a finite cyclic $\ell$--group. Moreover  $v_{\ell} (\sharp H_n)$ equals the smallest nonnegative integer $m$ such that $\zeta^{\ell^m}\in (K^\times)^{\ell^n}$. The roots of unity in $ (K^\times)^{\ell^n}\cap \mu_{\ell^\infty}$ are those of order dividing $\ell^{\max(z-n,0)}$ while $\zeta$ has order $\ell^{\max(h_j-\delta_j,0)}$. We then have $$v_\ell (\sharp H_n) = \max(\max(h_j-\delta_j,0)-\max(z-n,0),0)$$
and we may recover the formula in the statement because $h_j-\delta_j = \max_i(h_i-\delta_i)$. We are left to prove the isomorphism. Consider the surjective homomorphism $$\varphi:\mathbb{Z}^r\to \frac{G}{G_n}:(x_1,\ldots,x_r)\mapsto b_1^{x_1}\cdots b_r^{x_r}\mod G_n.$$Writing $J = \bigoplus_i \ell^{\delta_i}\mathbb{Z}$ we have $H_n = \varphi(J)$ and we know that the induced map $$\overline{\varphi}:\frac{\mathbb{Z}^r}{\ker(\varphi)+J}\to \frac{\varphi(\mathbb{Z}^r)}{\varphi(J)} = \frac{G/G_n}{H_n}$$ is an isomorphism.  It suffices to show that $\ker(\varphi)\subseteq J$, so suppose that $x_1,\ldots,x_r$ are integers such that $b_1^{x_1}\cdots b_r^{x_r}\in G_n$. Then $\prod_i \B_i^{x_i\ell^{z+d_i}} = \prod_i b_i^{x_i\ell^z}\in (K^\times)^{\ell^{z+n}}$, so Lemma \ref{lem:power-indep} yields that each $x_i\ell^{z+d_i}$ is divisible by $\ell^{z+n}$, i.e., each $x_i$ is divisible by $\ell^{\max(n-d_i,0)} = \ell^{\delta_i}$. This implies that $(x_1,\ldots,x_r)\in J$, so $\ker(\varphi)\subseteq J$.
\end{proof}

If $n>z+\max(d_1,\ldots,d_r)$ then $h_i-\delta_i = h_i-(n-d_i) < 0$, so $H_n$ is trivial. We hence have 

\begin{cor}\label{unique-d}
If $(d_1,\ldots,d_r,h_1,\ldots,h_r)$ expresses the $\ell$--divisibility of $G$ then $$\frac{G}{G_n} \cong \bigoplus_{i=1}^r  \frac{\mathbb{Z}}{\ell^{n-d_i}\mathbb{Z}}$$ for all sufficiently large $n$. In particular the $d$--parameters are uniquely determined by $(K,\ell,G)$ if we suppose (without loss of generality) that $d_1\leqslant\cdots\leqslant d_r$. Moreover, the ranks of the groups $G/G_n$ and $G$ are the same for all sufficiently large $n$.
\end{cor}

The $h$--parameters are not unique (not even if $G$ has rank $1$, see Example~\ref{single}) however the parameters appearing in Theorem~\ref{theorem2} only depend on $(K,\ell,G)$ and are uniquely determined.

\begin{exa}
Suppose that $G=\langle a \rangle$ where $a = A^{\ell^d}\zeta$ such that $A$ is strongly $\ell$--indivisible and $\zeta$ is a root of unity of order $\ell^h$. Theorem~\ref{thm-GmodGn} implies that for all $n\geqslant 1$ we have $$v_\ell \big(\sharp(G/G_n)\big) = \max(h,z-d,z-n,n-d)+\min(n-z,0)\,.$$
\end{exa}

\section{The degree of Kummer extensions} \label{section-degree}

In this section we want to find a parametric formula for the degree of the Kummer extensions $K_{\ell^m}(\sqrt[\ell^n]{G})/K_{\ell^m}$ where $m\geqslant n>0$. Kummer theory tells us that these degrees are powers of $\ell$, so we want to compute their $\ell$--adic valuation.

\begin{thm}\label{logdeg} Let $G$ be a torsion--free subgroup of $K^\times$ of finite rank $r>0$.
Suppose that $\ell$ is odd or $\zeta_4\in K$. Let $t\geqslant 1$ be the largest integer such that $K_\ell = K_{\ell^t}$. Let $m$ and $n$ be positive integers and suppose that $m\geqslant \max(n,t)$. Then the extension degree $\left[K_{\ell^m}(\sqrt[\ell^n]{G}):K_{\ell^m}\right]$ is a power of $\ell$ and 
 $$\begin{aligned} v_\ell \left[K_{\ell^m}(\sqrt[\ell^n]{G}):K_{\ell^m}\right] = & \max(h_1+n_1,\ldots,h_r+n_r,m)-m + rn - n_1 - \cdots  - n_r \end{aligned}$$ where $(d_1,\ldots,d_r,h_1,\ldots,h_r)$ expresses the $\ell$--divisibility of $G$ in $K$ and $n_i = \min(n,d_i)$.
\end{thm}
\begin{proof}
By Proposition~\ref{prop:stay-indiv}, strongly $\ell$--independent elements of $K$ remain strongly $\ell$--independent in $K_{\ell^m}$, so $(d_1,\ldots,d_r,h_1,\ldots,h_r)$ also expresses the $\ell$--divisibility of $G$ in $K_{\ell^m}$. The degree $\left[K_{\ell^m}(\sqrt[\ell^n]{G}):K_{\ell^m}\right]$ equals the index of $G\cap (K_{\ell^m}^\times)^{\ell^n}$ in $G$, so Theorem~\ref{thm-GmodGn} and the fact that $v_\ell(\sharp(K^\times_{\ell^m})_{\tors}) = m$ yield that this degree is a power of $\ell$ and that, using the notation $\delta_i = \max(n-d_i,0)$, its $\ell$--adic valuation equals $$\max(0,\max_i(h_i-\delta_i+n-m))+\delta_1+\cdots+\delta_n.$$
This is another way to write the formula in the statement.
\end{proof}

We now handle the remaining case not treated by the theorem: $\ell = 2$ and $\zeta_4\notin K$. If $m\geqslant 2$ then one has a formula using parameters expressing the $2$--divisibility of $G$ in $K_4$. The only case left is $m = n = 1$:

\begin{lem}\label{two} We have $[K(\sqrt{G}):K]=e\cdot [K_{4}(\sqrt{G}):K_4]$
where $e=2$ if $G$ contains minus a square in $K^\times$ and $e=1$ otherwise. 
\end{lem}

\begin{proof}
Kummer theory tells us that the quotient $[K(\sqrt{G}):K]/[K_4(\sqrt{G}):K_4]$ equals the cardinality of the quotient group $$\mathcal{G} = \frac{G\cap (K_4^\times)^2}{G\cap (K^\times)^2}.$$ To compute the cardinality of $\mathcal{G}$, we first note that any element of $G\cap (K_4^\times)^2$ is not strongly $2$--indivisible in $K$. Indeed, if $a\in G\cap (K_4^\times)^2$ would be strongly $2$--indivisible in $K$, then $K(\sqrt{a})$ would be a quadratic extension of $K$ contained in $K_4$, so $K(\sqrt{a}) = K_4 = K(\sqrt{-1})$ and in particular Kummer theory implies that 
 $a = -x^2$ for some $x\in K^\times$,
 contradicting the assumption that $a$ is strongly $2$--indivisible in $K$.

So any element of $G\cap (K_4^\times)^2$ is either a square in $K^\times$ or minus a square in $K^\times$. In the first case the element yields the trivial class in $\mathcal{G}$. This already shows that $e = 1$ if $G$ does not contain minus a square in $K^\times$ and that otherwise $e\geqslant 2$ (indeed, minus a square can not be a square in $K^\times$ because $\zeta_4\not\in K$). Now suppose that $x$ and $y$ are nontrivial elements of $\mathcal{G}$. Since they come from minus squares in $K^\times$, their products $x^2$ and $xy$ are trivial in $\mathcal{G}$, and hence $x = y$. This shows that there can be at most one nontrivial element in $\mathcal{G}$ and hence $e\leqslant 2$. 
\end{proof}

\begin{exa} 
With the notation of Theorem~\ref{logdeg}, suppose that $G=\langle a \rangle$ where $a = A^{\ell^d}\zeta$ such that $A$ is strongly $\ell$--indivisible and $\zeta$ is a root of unity of order $\ell^h$. Then we have 
 $$\begin{aligned}v_\ell \left[K_{\ell^m}(\sqrt[\ell^n]{a}):K_{\ell^m}\right] & =  \max(h+\min(n,d),m)-m + n - \min(n,d) \\ & = \max(h+n-m,n-d,0)\,.\end{aligned}$$
\end{exa}

\begin{exa}
With the notation of Theorem~\ref{logdeg}, suppose that $\zeta_\ell\notin K$. Then $h_i = 0$ hence $h_i+n_i \leqslant m$ and we have: 
$$v_\ell \left[K_{\ell^m}(\sqrt[\ell^n]{G}):K_{\ell^m}\right] = m-m + rn - n_1 - \cdots  - n_r = \sum_{i=1}^r \max(n-d_i,0).$$
\end{exa}

\begin{exa} With the notation of Theorem~\ref{logdeg}, suppose that $G$ has a basis consisting of strongly $\ell$--independent elements. Then $d_i=h_i=0$ hence
the extension $K_{\ell^m}(\sqrt[\ell^n]{G})$ of $K_{\ell^m}$ has the largest possible degree, namely $\ell^{nr}$.
\end{exa}

\section{The existence of the density} \label{section-existence}

The natural density of a set $S$ of primes of $K$ is $\lim_{n\rightarrow \infty} \sharp (S\cap \mathcal P_n)/ \sharp (\mathcal P_n)$, where $\mathcal P_n$ is the set of primes of $K$ having residue field of size at most $n$. By the upper and lower density we respectively mean the limit inferior and superior: these exist and if they coincide then the density is well-defined.

If $G\subset K^\times$ is a finitely generated subgroup, we neglect the finitely many primes $\p$ of $K$ for which $(G \bmod \p)$ is not well-defined or contains $(0 \bmod \p)$. We write $\D(G)$ for the density of the set of primes $\p$ of $K$ such that $(G \bmod \p)$ has exponent (equivalently, size) prime to $\ell$.

\begin{thm}\label{thmdensity} Let $G\subset K^\times$ be a finitely generated subgroup. The set of primes $\p$ of $K$ such that $(G \bmod \p)$ has size prime to $\ell$ has a natural density, and this is given by the formula:
\begin{equation*}
\D(G)=\sum_{n\geqslant 0}\;\big(\; [K_{\ell^n}(\sqrt[\ell^{n}]{G}):K]^{-1}-[K_{\ell^{n+1}}(\sqrt[\ell^{n}]{G}):K]^{-1}\; \big)\,.
\end{equation*}
\end{thm}
\begin{proof} This proof is inspired from \cite[Theorem 3.2]{Jones_Rouse}.
We neglect the finitely many primes $\p$ of $K$ such that $\p$ ramifies in the extension $K_{\ell^n}(\sqrt[\ell^{n}]{G})$ for some $n\geqslant 0$. Let $S$ be the set of primes $\p$ of $K$ such that $(G \bmod \p)$ has size prime to $\ell$ or, equivalently, such that for every $a\in G$ the reduction $(a \bmod \p)$ has some $\ell^n$-th root in the residue field $k_\p$ for every $n\geqslant 0$. 
We can write $S$ as a disjoint union of sets $S_n$ for $n\geqslant 0$ by intersecting $S$ with the primes $\p$ of $K$ such that $v_{\ell}(\sharp k_\p^\times)=n$. The set  $S_n$ consists of the primes of $K$ that split completely in $K_{\ell^n}(\sqrt[\ell^{n}]{G})$ but not in $K_{\ell^{n+1}}(\sqrt[\ell^{n}]{G})$. The Chebotarev Density Theorem implies that $S_n$ has a natural density, given by: 
$$\dens(S_n)= [K_{\ell^n}(\sqrt[\ell^{n}]{G}):K]^{-1}-[K_{\ell^{n+1}}(\sqrt[\ell^{n}]{G}):K]^{-1}\,.$$ 
We are left to show that 
$\dens(S)$ exists  and is equal to the sum of $\dens(S_n)$ for $n\geq 0$.
Since the $S_n$ are pairwise disjoint subsets of $S$, the lower density of $S$ is at least the sum of $\dens(S_n)$. To show that the upper density of $S$ is at most the sum of $\dens(S_n)$, remark that $\cup_{n\geqslant N} S_n$ is contained in the set of primes of $K$ which split completely in $K_{\ell^{N}}$, and that this set has a density going to zero for $N$ going to infinity, by the Chebotarev Density Theorem.
\end{proof}

\begin{rem} For every $n\geqslant 1$, the set of primes $\p$ of $K$ such that the size of $(G \bmod \p)$ has $\ell$-adic valuation $n$ has a natural density, that is  equal to $\D(G^{\ell^{n}})-\D(G^{\ell^{n-1}})$.
\end{rem}

By the inclusion--exclusion principle we also have:

\begin{rem}
Let $g_i\in K^{\times}$ and let $n_i\geqslant 0$
for $i=1,\ldots, r$. The set of primes $\p$ of $K$ such that the multiplicative order of $(g_i \bmod \p)$ has $\ell$-adic valuation $n_i\,$ for every $i$ has a natural density, and this is given by the formula:
$$\D( \langle g_1^{\ell^{n_1}}, \ldots, g_r^{\ell^{n_r}} \rangle )+\sum_{k = 1}^{r} (-1)^{k}\cdot  \! \!\!\!\! \sum_{1 \leqslant i_{1} < \cdots < i_{k} \leqslant r} \D( \langle g_1^{\ell^{n_1}}, \ldots, g_r^{\ell^{n_r}}, g_{i_{1}}^{\ell^{n_{i_{1}}-1}}, \ldots, g_{i_{k}}^{\ell^{n_{i_{k}}-1}}\rangle)\,. $$
 \end{rem}

\section{A closed-form expression for the density}\label{section-proof}

While proving Theorem~\ref{theorem2}, we also show  the following:

\begin{rem}\label{tau} With the notation of Theorems~\ref{theorem2}~and~\ref{thm-def-parameters}, the parameter $\tau$ is the maximum of $\{t\}\cup\{h_i+d_i :h_i> 0\}$. In particular, $t\leqslant \tau \leqslant t+\max(d_1,\ldots,d_r)$.
\end{rem}

\emph{Proof of Theorem~\ref{theorem2}:} Let $(d_1,\ldots,d_r, h_1,\ldots,h_r)$ be parameters for the $\ell$--divisibility of $G$ in $K$. We may assume without loss of generality that $d_1\leqslant\cdots\leqslant d_r$. Remark~\ref{unique-d} shows that (b) holds. 
We want to compute the infinite sum given in  Theorem~\ref{thmdensity}. The summand for $n=0$ equals $1-{[K_\ell:K]}^{-1}$ because the Kummer extensions are trivial. For $1\leqslant n<t$, the $n$--th summand disappears because $K_{\ell^n} = K_{\ell^{n+1}}$. We  are left with the following expression:
$$\D(G) = 1-\frac{1}{[K_\ell:K]} + \sum_{n\geqslant t} \left([K_{\ell^n}(\sqrt[\ell^{n}]{G}):K]^{-1}-[K_{\ell^{n+1}}(\sqrt[\ell^{n}]{G}):K]^{-1}\right).$$

Let $m\geqslant n\geqslant t$. By Theorem~\ref{logdeg} and the formula $[K_{\ell^m}:K] = \ell^{m-t}[K_\ell:K]$ we have $[K_{\ell^m}(\sqrt[\ell^{n}]{G}):K] = [K_\ell:K]\cdot \ell^{L(m,n)}$, where 
$$L(m,n) = \max(h_1+n_1,\ldots,h_r+n_r,m)-t+\sum_{i=1}^r \max(n-d_i,0)$$
in which $n_i = \min(n,d_i)$. 
Define $\delta(n)=L(n+1,n)-L(n,n)$ for $n\geqslant t$. Note that $\delta(n) = 0$ if $n< \max_i(h_i+n_i)$ and that otherwise we have $$L(n,n) = n-t+\sum_i \max(n-d_i,0) = L(n+1,n)-1.$$ If we denote by $\Delta$ the set of integers $n\geqslant t$ such that $\delta(n) \neq 0$, then we have
 $$\D(G) = 1 - \frac{1}{[K_\ell:K]} + \frac{\ell^{t-1}(\ell-1)}{[K_\ell:K]}\sum_{n\in \Delta} \ell^{-n-\sum_{i=1}^{r} \max(n-d_i,0)}\,.$$

We claim that $\Delta = \mathbb{Z}\cap  [\tau,+\infty)$, where $\tau = \max(\{t\}\cup\{h_i+d_i :h_i> 0\})$. For one inclusion: if $n\in \Delta$ then the inequalities $n\geqslant t$ and $n\geqslant h_i+n_i$ hold for each $i$. If $h_i> 0$, we must have $n_i=d_i$ thus $n\geqslant h_i+d_i$ holds and we deduce $n\geqslant \tau$. For the other inclusion, if $n\geqslant\tau$ then $n\geqslant t$, so we are left to show $n\geqslant h_i+n_i$. If $h_i = 0$ this is trivial, and if $h_i> 0$ we have $n\geqslant \tau\geqslant h_i+d_i$. The claim is proven.

We have shown that $\tau$ is the smallest  integer $n\geqslant 1$ satisfying $n\geqslant t$ and $\delta(n)\neq 0$, which is equivalent to $n\geqslant t$ and $K_{\ell^{n+1}}(\sqrt[\ell^n]{G})\neq K_{\ell^{n}}(\sqrt[\ell^n]{G})$. So $\tau$ is the smallest integer $n\geqslant 1$ such that $\zeta_{\ell^{n+1}}\not\in K_{\ell^n}(\sqrt[\ell^n]{G})$. Remark~\ref{tau} and (a) are proven, and we have $\tau\leqslant \tau_1\leqslant\cdots\leqslant \tau_r$.

For any subset $J$ of $\mathbb{R}$ we write $\sigma_J = 0$ if $J\cap\mathbb{Z}=\emptyset$ and otherwise $$\sigma_J = \sum_{n\in J\cap\mathbb{Z}} \ell^{-n-\sum_{i=1}^{r} \max(n-d_i,0)}$$ We want to compute $\sigma_{[\tau,+\infty)}$. Note that the following equalities hold: $$\sigma_{[\tau,\tau_1)} = \frac{\ell^{-\tau}-\ell^{-\tau_1}}{1-\ell^{-1}}\qquad\text{and}\qquad \sigma_{[\tau_r,+\infty)} = \ell^{d_1+\cdots+d_r}\frac{\ell^{-(r+1)\tau_r}}{1-\ell^{-r-1}}.$$
Indeed, the first one is true if $\tau = \tau_1$ and otherwise $\tau \neq \tau_1$ implies that $\tau < d_1$ and hence any $n\in[\tau,\tau_1)$ satisfies $\sum_i \max(n-d_i,0) = 0$. The second one is true because any $n\geq\tau_r$ satisfies $n\geqslant \tau_r\geqslant d_i$ for any $i$, so $\sum_i \max(n-d_i,0) = rn-d_1-\cdots-d_r$. Similarly, for any $1\leqslant j<r$ we have $$\sigma_{[\tau_j,\tau_{j+1})} = \ell^{d_1+\cdots+d_j}\frac{\ell^{-(j+1)\tau_j}-\ell^{-(j+1)\tau_{j+1}}}{1-\ell^{-j-1}}$$ because this is true if $\tau_j = \tau_{j+1}$ and otherwise we have $\tau_{j+1} = d_{j+1}$, so $\sum_i \max(n-d_i,0) = jn-d_1-\cdots-d_j$ for any $n\in[\tau_j,\tau_{j+1})$. Adding up the contributions we find that
$$\begin{aligned}\sigma_{[\tau,+\infty)} & = \frac{\ell^{-\tau}-\ell^{-\tau_1}}{1-\ell^{-1}}+\ell^{d_1+\cdots+d_r}\frac{\ell^{-(r+1)\tau_r}}{1-\ell^{-r-1}}+\sum_{s=1}^{r-1}\ell^{d_1+\cdots+d_s}\frac{\ell^{-(s+1)\tau_s}-\ell^{-(s+1)\tau_{s+1}}}{1-\ell^{-s-1}} \\ & = \frac{\ell^{-\tau}}{1-\ell^{-1}}+\sum_{s=1}^r \ell^{d_1+\cdots+d_s}\frac{\ell^{-(s+1)\tau_s}}{1-\ell^{-s-1}} - \sum_{s=0}^{r-1}\ell^{d_1+\cdots+d_{s+1}-d_{s+1}}\frac{\ell^{-(s+1)\tau_{s+1}}}{1-\ell^{-s-1}} \\ & = \frac{\ell^{-\tau}}{1-\ell^{-1}}+\sum_{s=1}^r \ell^{d_1+\cdots+d_s}\left(\frac{\ell^{-(s+1)\tau_s}}{1-\ell^{-s-1}} - \frac{\ell^{-s\tau_s-d_s}}{1-\ell^{-s}}\right)\end{aligned}$$thus recovering the formula of (c).
\hfill$ \square$

Note, since $\sigma_{[\tau,+\infty)}$ is strictly positive then the density $\D(G)$ is greater than $1 - \frac{1}{[K_\ell:K]}$. The formula of (c) shows that the density is smaller than one: the bracketed difference inside the sum is strictly positive so $\D(G)$ is smaller than  
$$1 - \frac{1}{[K_\ell:K]} + \frac{\ell^{t-1}(\ell-1)}{[K_\ell:K]} \frac{\ell^{-\tau}}{1-\ell^{-1}}=1 - \frac{1}{[K_\ell:K]} + \frac{\ell^{t-\tau}}{[K_\ell:K]}\leq 1 \,.$$

\begin{proof}[Proof of Theorem~\ref{theorem3}]
In the formula of Theorem~\ref{thmdensity} for $\ell=2$, the summand for $n=0$ gives no contribution because $K_2= K$, so we have:
$$D_{K,2}(G)=\sum_{n\geqslant 1}\;\big(\; [K_{2^n}(\sqrt[2^{n}]{G}):K]^{-1}-[K_{2^{n+1}}(\sqrt[2^{n}]{G}):K]^{-1}\; \big)\,.$$
Over $K_4$ we analogously have:
$$ D_{K_4,2}(G)=\sum_{n\geqslant 2}\;\big(\; [K_{2^n}(\sqrt[2^{n}]{G}):K_4]^{-1}-[K_{2^{n+1}}(\sqrt[2^{n}]{G}):K_4]^{-1}\; \big)\,.$$
Since $[K_4:K]=2$ holds, the difference  $D_{K,2}(G)-\frac{1}{2} D_{K_4,2}(G)$ is exactly the summand for $n=1$ of $D_{K,2}(G)$, namely 
$[K(\sqrt{G}):K]^{-1}-\frac{1}{2}[K_{4}(\sqrt{G}):K_4]^{-1}$. We conclude by Lemma~\ref{two}.
\end{proof}

\section{Heuristics and special cases}\label{specialheuristics}

\subsection{Heuristics}\label{heuristics}
We keep the notation of Theorem~\ref{theorem2}, and explain the formula for the density with heuristics,
by assuming that the $\ell$--part of the reductions of strongly $\ell$--independent elements of $K^\times$ behave ``randomly''.
Let $S$ be the set of primes $\p$ of $K$ such that the size of $(G\bmod \p)$ is well-defined and is prime to $\ell$. For $n\geqslant 0$ let $S_n=S\cap P_n$, where $P_n$ consists of the primes $\p$ of $K$ such that $v_{\ell}(\sharp k_\p^\times)=n$.
As seen in the proof of 
Theorem~\ref{thmdensity}, we have:
$$\D(G) = \dens(S)=\sum_{n\geqslant 0}\dens(S_n)\,.
$$
We have computed $\dens(S_0)=1-[K_\ell:K]^{-1}$, and in fact $S_0=P_0$ holds.
For $1\leq n< t$ we have computed $\dens(S_n)=0$, and this can be explained by noticing that $P_n=\emptyset$.
For $n\geq t$ we have computed (see the proof of Theorem~\ref{theorem2}) that   
$$\dens(S_n) =\frac{\ell^{t-1-n}(\ell-1)}{[K_\ell:K]}\cdot \delta(n)\cdot \ell^{-\sum_i \max(n-d_i,0)}$$ where $\delta(n) = 1$ if $n\in \Delta$ and $\delta(n) = 0$ otherwise. The first factor is $\dens(P_n)$ by the Chebotarev Density Theorem. To understand $\dens(S_n)$, we conclude by argumenting that for $\p\in P_n$ the ``probability'' that $\p\in S$ is the product of the remaining two factors.

By Theorem~\ref{thm-def-parameters}, there is a basis of $G$ given by $b_i=\B_i^{\ell^{d_i}}\zeta_i$ for $i=1,\ldots, r$, where $B_1,\ldots,B_r$ are strongly $\ell$--independent and $\zeta_i$ is a root of unity of order $\ell^{h_i}$. We are supposing  $v_{\ell}(\sharp k_\p^\times)=n\geq t$ and we want the order of $(b_i \bmod\p)$ to be prime to $\ell$ for every $i$. We distinguish two cases: 

\begin{itemize}
\item If $h_i=0$, we want the $\ell$--part of $(B_i \bmod\p)$ to have order at most $\ell^{d_i}$: the probability is $\ell^{- \max(n-d_i,0)}$ because if $n\leqslant d_i$ it is $1$ while if $n>d_i$ there are $\ell^{d_i}$ favorable outcomes over $\ell^{n}$ outcomes.
\item If $h_i >0$ we want the $\ell$--part of $(B_i \bmod\p)$ to be some $\ell^{d_i}$-th root of 
$(\zeta^{-1}_i \bmod\p)$. The probability is $0$ if $h_i+d_i>n$ (there are no such roots in $k_\p^\times$) while if $h_i+d_i\leqslant n$ the probability is $\ell^{- \max(n-d_i,0)}=\ell^{d_i-n}$ (there are $\ell^{d_i}$ such roots among $\ell^n$ elements).
\end{itemize}

The joint probability is then $\ell^{- \sum_i \max(n-d_i,0)}$ unless $h_i+d_i> n$ for some $i$ such that $h_i> 0$ holds, in which case the probability is zero. Since $n\geq t$ and by Remark~\ref{tau}, the probability is zero if and only if $n\notin \Delta=\mathbb{Z}\cap  [\tau,+\infty)$ therefore we exactly need the factor $\delta(n)$. We have recovered the formula of Theorem~\ref{theorem2}.

\subsection{Special cases}\label{specialcases} Under some additional conditions, the formula of Theorem~\ref{theorem2} simplifies:

{\bf Rank $1$:} Suppose that $G=\langle a \rangle$ where $a = A^{\ell^d}\zeta$ such that $A$ is strongly $\ell$--indivisible and $\zeta$ is a root of unity of order $\ell^h$.
If $h>0$ then in particular $K_\ell = K$ holds and we have $\tau = \max(t,h+d)$ and $\tau_1 = \max(\tau,d)=\tau$ so Theorem~\ref{theorem2} gives:
$$\D(G) = \ell^{t+d-2\tau+1}/(\ell+1)\,.$$
If $h = 0$ then $\tau = t$ and $\tau_1 = \max(\tau,d)= \max(t, d)$ so we have:
$$\D(G) = 1-\frac{1}{[K_\ell:K]}\left(\ell^{\min(0,t-d)}-\frac{\ell^{1-|t-d|}}{\ell+1}\right)\,.$$
In particular, if $K_\ell = K$ and $h=0$ hold, we have: $$\D(G) = \left\{\begin{array}{ll} \frac{\ell^{1+d-t}}{\ell+1} & \text{if }d\leqslant t \\ 1-\frac{\ell^{t-d}}{\ell+1} & \text{if }d> t \end{array}\right.\,.$$
If $K_\ell = K$ and $h>0$ hold, the parameters $(d, h)$ may be replaced by $(d,0)$ as soon as $h+d\leqslant t$ because in this case $\zeta$ is an $\ell^d$--th power. The two choices for the parameters give of course the same density.\smallskip

{\bf The case $\tau = t$:} We have $\ell^{t-1}(\ell-1)\ell^{-\tau} = 1-\ell^{-1}$ hence Theorem~\ref{theorem2} gives: $$\D(G) = 1-\frac{\ell^{t-1}(\ell-1)}{[K_\ell:K]}\sum_{i=1}^r \ell^{d_1+\cdots+d_i-i\tau_i}\left(\frac{\ell^{-d_i}}{1-\ell^{-i}}-\frac{\ell^{-\tau_i}}{1-\ell^{-i-1}}\right).$$
By Remark~\ref{theorem2}, we have $\tau = t$ as soon  as $h_i+d_i\leqslant t$ holds for every $i$ such that $h_i> 0$. This is the case if $h_i=0$ for every $i$ so for example if we have $K_\ell \neq K$.\smallskip

{\bf The case $h_i=0$ and $d_i\geqslant t$ for all $i$:} We have $\tau = t$ and $d_i= \tau_i$ so the formula of the previous case becomes:
$$\D(G) = 1-\frac{\ell^{t-1}(\ell-1)}{[K_\ell:K]}\sum_{i=1}^r \ell^{d_1+\cdots+d_{i-1}-id_i}\left(\frac{1}{1-\ell^{-i}}-\frac{1}{1-\ell^{-i-1}}\right).$$
Since $d_1+\cdots+d_{i-1}-id_i\leq -d_i$ holds, if all  $d_i$ are large then the density is close to $1$. This has to be expected because under the given assumptions we have $G=H^{\ell^N}$ for some subgroup $H$ of $K^\times$ and for some large positive integer $N$.\smallskip

{\bf The case $G\cap (K^\times)^\ell = G^\ell$:}  This assumption means that $G$ has a basis consisting of strongly $\ell$--independent elements, so we have $d_i=h_i=0$ for every $i$. By Remark~\ref{tau}, we then have $\tau=\tau_i=t$.
In the formula of Theorem~\ref{theorem2} we have a telescopic sum
 $$1-\frac{1}{[K_\ell:K]}+\frac{\ell^{t-1}(\ell-1)}{[K_\ell:K]}\left(\frac{\ell^{-t}}{1-\ell^{-1}}-\sum_{i=1}^r \frac{\ell^{-it}}{1-\ell^{-i}}-\frac{\ell^{-(i+1)t}}{1-\ell^{-(i+1)}}\right)$$
thus we recover the formula of Theorem~\ref{theorem1}. For $r$ large, the density is close to the lower bound $1-[K_\ell:K]^{-1}$ and indeed a higher rank means that the reductions of many points need to have order prime to $\ell$. As a curiosity, we may rewrite the formula of Theorem~\ref{theorem1} as follows:
$$\D(G) = 1-\frac{1}{[K_\ell:K]} + \frac{[K\cap\mathbb{Q}_{\ell^\infty}:\mathbb{Q}]}{(\ell^{r+1}-1)\cdot [K_\ell\cap\mathbb{Q}_{\ell^\infty}:\mathbb{Q_\ell}]^{r+1}}\,.$$\smallskip

\section{Computations}\label{section-computations}

In this section we compute parameters and densities in a few examples, testing our results with Sage~\cite{sage}.
First note that one can compute parameters as in Lemma~\ref{convenience}: there are algorithms to decide whether some given element is an $\ell^d$--th power for some nonnegative integer $d$ (e.g. implemented in Sage). In particular, one can check the strong $\ell$--indivisibility of a given element and the strong $\ell$--independence of a given set of elements (restrict the possible exponents in the definition of strong $\ell$--independence to $\{0,1,\ldots,\ell-1\}$). So now the proof of Theorem~\ref{thm-def-parameters} gives us an algorithm to compute parameters expressing the $\ell$--divisibility of $G$ in $K$. Indeed, start with a basis $\mathcal{B}$ of $G$ and check whether the associated strongly $\ell$--indivisible elements are strongly $\ell$--independent. If they are, we are done, and if they are not, we can change the basis as prescribed in the proof, increasing $d(\mathcal{B})$. Now start anew with this altered basis. This process has to stop since $d(\mathcal{B})$ is bounded from above.

\begin{exa}\label{exa1218}
Let $\ell = 3$ and $G = \langle 12,18\rangle\subset \mathbb Q^\times$.  Strongly $3$--indivisible elements associated to the given basis are $12$ and $18$, which are not strongly $3$--independent because $12\cdot 18 = 6^3$. This relation yields  $G = \langle 18,6^3\rangle$. Since $18$ and $6$ are strongly $3$--independent, $(0,1,0,0)$ expresses the $3$--divisibility of $G$. Let us now work over $K = \mathbb{Q}(\zeta_9+\zeta_9^{-1})$. The elements $18$ and $6$ remain strongly $3$--independent in $\mathbb Q_9$ hence also in $K$. Thus $G$ has the same parameters over $K$.
\end{exa}

\begin{exa} \label{exastart} Let $K=\mathbb Q(\zeta_8)$, $\ell = 2$ and consider the following subgroups of $K^\times$: $G_1=\langle 12,18\rangle$,   $G_2=\langle 12\zeta_8,18\rangle$ and $G_3=\langle 12,18\zeta_8\rangle$. Note that $2 = a^2$ where $a = \zeta_8^3-\zeta_8$. Since $3\mathcal{O}_K$ is not the square of an ideal (because $3$ is unramified in $K$) and the primes dividing $a\OO_K$ and $3\OO_K$ are distinct, we know that $3$ and $3a$ are strongly $2$--indivisible (they are not even a square times a {unit}). Moreover, note that $a/(1-\zeta_8)^2 = -\zeta_8^3-\zeta_8^2-\zeta_8$ is a fundamental unit, as can be checked with Sage. In particular, $a$ is strongly $2$--indivisible and together with the previous remarks this shows that $3$ and $a$ are strongly $2$--independent in $K$. So $3a^4 = 12$ and $3a$ are strongly $2$--independent and hence $(0,1,0,0)$ are parameters for $G_1 = \langle 3a^4,(3a)^2\rangle$. Similarly, $G_2 = \langle 3a^4\zeta_8,(3a)^2\rangle$ has parameters $(0,1,0,0)$ and $G_3 = \langle 3a^4,\zeta_8(3a)^2\rangle$ has parameters $(0,1,0,3)$.
\end{exa}

To compute the densities, we first compute parameters expressing the $\ell$--divisibility of $G$ in $K$ (as described above) and then use Theorem~\ref{theorem2} and Remark~\ref{tau}. If $\ell = 2$ and $\zeta_4\notin K$, we use  Theorem~\ref{theorem3} instead: we should compute $D_{K_4,2}(G)$ (with the above methods) and $[K_4(\sqrt{G}):K_4]$ (using Theorem~\ref{logdeg}), and we should decide whether $G$ contains minus a square in $K^\times$ or not, which can be done by considering the products $\prod g_i^{\delta_i}$, where $g_1,\ldots,g_r$ are generators of $G$ and $\delta_i\in\{0,1\}$.

\begin{exa}
Let $\ell = 3$ and let $G = \langle 12,18\rangle\subset \mathbb Q^\times$. Since $\zeta_9\not\in \mathbb{Q}_3$ holds, we have $t = 1$. By Example~\ref{exa1218},  
we have parameters $(0,1,0,0)$ hence $\tau_1 = \tau_2 = 1$. We compute $D_{\mathbb{Q},3}(G) = 8/13$. Now work over $K = \mathbb{Q}(\zeta_9+\zeta_9^{-1})$, which has degree $3$ over $\Q$. Since $K_3=\mathbb Q_9$, we have $t=2$. By Example~\ref{exa1218}, we have parameters $(0,1,0,0)$ hence $\tau_1 = \tau_2 = 2$. We compute $D_{K,3}(G) = 20/39$.
We approximated both densities by testing with Sage the primes of norm at most $10^8$ (respectively, $5\cdot 10^5$). We found the values $3545696/5761455$ (respectively,  $21386/41507$) hence an error of less than $0,3\%$.
\end{exa}

\begin{exa} We finally explain the example at the beginning of the introduction. We have $\ell = 2$ and $K=\mathbb{Q}(\zeta_8)$ hence $t = 3$. By Example~\ref{exastart}, the parameters for $G_1$ and $G_2$ are the same, namely $(0,1,0,0)$ hence $\tau_1=\tau_2=3$. On the other hand  the parameters for $G_3$ are $(0,1,0,3)$ hence $\tau_1=3$ and $\tau_2=4$. This allows us to compute $D_{K,2}(G_1)=D_{K,2}(G_2)=\frac{1}{56}$ and $D_{K,2}(G_3)=\frac{1}{448}$. The difference between the densities relies on the fact that $18$ is a square in $K^\times$ while $12$ is not.
We approximated these densities by testing with Sage the primes of norm at most $10^6$ (respectively, $10^6$ and $45\cdot 10^5$).
We found the values $1412/78469$ (resp. $1388/78469$ and $705/316281$) hence an error of less than $0,02\%$.
\end{exa}


\begin{thebibliography}{10} \expandafter\ifx\csname url\endcsname\relax   \def\url#1{\texttt{#1}}\fi \expandafter\ifx\csname urlprefix\endcsname\relax\def\urlprefix{URL }\fi

\bibitem{Birch} B.~J.~Birch,  \emph{Cyclotomic fields and Kummer extensions} in  Algebraic Number Theory, edited by J.W.S.~Cassels and A.~Fr\"ohlich, Academic Press, London, 1967.

\bibitem{Hasse1} H. Hasse, \emph{\"Uber die Dichte der Primzahlen $p$, f\"ur die eine vorgegebene ganzrationale Zahl $a\neq 0$ von durch eine vorgegebene Primzahl $l\neq 2$ teilbarer bzw. unteilbarer Ordnung mod $p$ ist}, Math. Ann. {\bf 162} (1965/1966), 74--76.

\bibitem{Hasse2} H.~Hasse, \emph{\"Uber die Dichte der Primzahlen $p$, f\"ur die eine vorgegebene ganzrationale Zahl $a\neq 0$ von gerader bzw. ungerader Ordnung mod  $p$ ist}, Math. Ann. {\bf 166} (1966), 19--23.

\bibitem{Jones_Rouse}
R.~Jones and J.~Rouse, \emph{Iterated endomorphisms of abelian algebraic groups}, Proc. Lond. Math. Soc. \textbf{100} (2010), no.~3, 763--794. 

\bibitem{Lenstra}
H.~W.~Jr.~Lenstra, \emph{Commentary on H: Divisibility and congruences}. Andrzej Schinzel Selecta Vol.II, European Mathematical Society, Z{\"u}rich, 2007, 901--902. 

\bibitem{Lang}
S.~Lang, \emph{Algebra}, Graduate Texts in Mathematics 211, Springer-Verlag, New York, 2002.

\bibitem{MoreeArtin}
P.~Moree, \emph{Artin's primitive root conjecture--a survey}, Integers {\bf 12} (2012), no.~6, 1305--1416; arXiv:0412.262.

\bibitem{PeruccaDensity}  
A.~Perucca, \emph{On the reduction of points on abelian varieties and tori}, Int.~Math.~Res.~Notices, vol. 2011 (2011), no.~7, 293--308.

\bibitem{PeruccaKummer}  
A.~Perucca, \emph{The order of the reductions of an algebraic integer}, submitted for publication.

\bibitem{Schinzel}
A.~Schinzel, \emph{Abelian binomials, power residues and exponential congruences}, Acta Arith. \textbf{32} (1977), no.~3, 245--274.
Addendum, ibid. \textbf{36} (1980), 101--104. See also Andrzej Schinzel Selecta Vol.II, European Mathematical Society, Z{\"u}rich, 2007, 939--970.

\bibitem{sage}
W.~A.~Stein et al., Sage Mathematics Software (Version 5.7). The Sage Development Team, 2013, http://www.sagemath.org.

\end{thebibliography}
\end{document}